\documentclass[12pt,a4paper,reqno]{amsart}
\usepackage{amsmath,amsthm,verbatim,amssymb,amsfonts,amscd, graphicx}
\usepackage{xcolor}
\usepackage{geometry}
\usepackage{graphicx}
\usepackage{hyperref}
\usepackage{enumitem}
\usepackage{graphicx}
\usepackage{soul}
\usepackage[english,capitalize]{cleveref}
\definecolor{darkblue}{rgb}{0,0,0.3}
\definecolor{urlblue}{rgb}{0,0,0.7}
\hypersetup{
	colorlinks=true,
	citecolor=blue,
	linkcolor=darkblue,
	urlcolor=urlblue,
}

\newtheorem{lemma}{Lemma}
\newtheorem{theorem}{Theorem}

\newtheorem{prop}{Proposition}
\newtheorem{corollary}{Corollary}
\theoremstyle{definition}
\newtheorem{remark}{Remark}

\newcommand{\RR}{\mathbb{R}}

\DeclareMathOperator{\Ric}{Ric}
\renewcommand{\tilde}{\widetilde}
\newcommand{\D}{\nabla}
\newcommand{\p}{\partial}

\renewcommand{\L}{\mathcal{L}}
\renewcommand{\th}{\theta}
\newcommand{\metric}[2]{\langle#1,#2\rangle}

\renewcommand{\bar}[1]{\overline{#1}}
\newcommand{\tg}{\tilde{g}}

\newcommand{\tDelta}{{\tilde{\Delta}}}
\newcommand{\tR}{\tilde{R}}

\DeclareMathOperator{\tr}{tr}

\renewcommand{\leq}{\leqslant}
\renewcommand{\geq}{\geqslant}

\newcommand{\tRic}{{\tilde{\Ric}}}
\DeclareMathOperator{\Lip}{Lip}
\newcommand{\tw}{{\tilde{w}}}
\newcommand{\tu}{{\tilde{u}}}
\newcommand{\tth}{{\tilde{h}}}
\newcommand{\tA}{{\tilde{A}}}
\newcommand{\tH}{{\tilde{H}}}
\newcommand{\te}{{\tilde{e}}}
\DeclareMathOperator{\loc}{loc}

\newcommand{\bG}{\overline{G}}
\renewcommand{\SS}{\mathbb{S}}

\usepackage[maxbibnames=99,backend=biber, sorting=nyt, doi, url=false]{biblatex}
\title{Connected sum of manifolds with spectral Ricci lower bounds}

\author{Gioacchino Antonelli}
\address{Gioacchino Antonelli
\hfill\break Department of Mathematics, University of Notre Dame, Hurley Hall, 255 Hurley, Notre Dame, IN 46556, United States}
\email{gantonel@nd.edu}

\author{Kai Xu}
\address{Kai Xu
\hfill\break Department of Mathematics, University of California, Berkeley, CA 94720, USA,}
\email{kaixu@berkeley.edu}

\begin{document}

\begin{abstract}
    Let \( n > 2 \), \( \gamma > \frac{n-1}{n-2} \), and \( \lambda \in \mathbb{R} \). We prove that if \( M \) and \( N \) are two smooth \( n \)-manifolds that admit a complete Riemannian metric satisfying
    \[
    -\gamma\Delta + \mathrm{Ric} > \lambda,
    \]
    then the connected sum \( M \# N \) also admits such a metric. The construction geometrically resembles a Gromov-Lawson tunnel; the range \( \gamma > \frac{n-1}{n-2} \) is sharp for this to hold.
\end{abstract}

\maketitle

\section{Introduction}


On a smooth complete $n$-dimensional Riemannian manifold $(M^n,g)$ with $n\geq 2$, let the function $\Ric\in\Lip_{\loc}(M)$ be defined as
\[
\Ric(x):=\inf\Big\{\Ric_x(v,v): v\in T_xM,\,|v|=1\Big\}.
\]
For constants $\gamma\geq0,\lambda\in\mathbb R$, we say that \textit{$M$ satisfies $\lambda_1(-\gamma\Delta+\Ric)\geq \lambda$}
if any of the two following equivalent conditions holds (see \cite{Fischer-Colbrie-Schoen}):
\begin{enumerate}[label={(\roman*)}, topsep=2pt, itemsep=1pt]
    \item for all $\varphi\in C^1_c(M)$ it holds $\int_M\gamma|\D\varphi|^2+\Ric\cdot\varphi^2\geq \int_M \lambda\varphi^2$,
    \item there exists $u\in C^{2,\alpha}(M)$ such that $u>0$ and $-\gamma\Delta u+\Ric\cdot u\geq \lambda u$ on $M$. 
\end{enumerate}
This is a global condition on the Ricci curvature, and is (often strictly) weaker than pointwise Ricci lower bounds (see the examples in \cite{APX24, AX24}). See \cite{APX24, AX24, CatinoMariMastroliaRoncoroni, CarronMondelloTewodrose} for an introduction and recent results on manifolds satisfying the latter condition. The study of Ricci lower bounds in the spectral sense has recently been connected to the study of stable minimal hypersurfaces; we refer to \cite{ChodoshLiBernstein2, CLMS, Mazet}.

In \cite{AX24}, the authors proved the following volume comparison theorem: if
\begin{equation}\label{eq-intro:positivity}
    n\geq3,\quad \gamma\leq\frac{n-1}{n-2},\quad\text{$M$ is closed and satisfies }\lambda_1(-\gamma\Delta+\Ric)\geq n-1,
\end{equation}
then $|M|\leq|\SS^n|$. See \cite{AX24} for more discussions about this result. Here, the bound $\frac{n-1}{n-2}$ for $\gamma$ might draw the most curiosity. It is in fact a sharp threshold: in the supercritical range $\gamma>\frac{n-1}{n-2}$, there exist warped product metrics on $\SS^{n-1}\times\SS^1$ that satisfy $\lambda_1(-\gamma\Delta+\Ric)>0$, and then counterexamples to the volume bound above arise by taking large enough covering spaces: see \cite[Remark 4]{AX24} for more details about this. The same threshold also appeared in Bour--Carron \cite{BourCarron}, where a technique based on the Bochner formula was used to show that $\text{\eqref{eq-intro:positivity}}$ implies $b_1(M)\leq n$ (the relevant statement \cite[Proposition 3.5]{BourCarron} actually inspired the warped-product counterexample above). See also Li--Wang \cite{LiWang} where the same bound of $\gamma$ arises in a different context.

To summarize, the geometric and topological properties associated with the spectral Ricci lower bound appear to undergo a transition at the critical value $\gamma=\frac{n-1}{n-2}$. The aim of this note is to provide further evidence of this. Specifically, we prove the following Gromov--Lawson and Schoen--Yau type connected sum result:

\begin{theorem}\label{thm:tunnel}
    Let $n\geq3$, $\gamma>\frac{n-1}{n-2}$, and $\lambda\in\RR$. Let $(M^n,g)$ be a (possibly non-connected) complete manifold satisfying $\lambda_1(-\gamma\Delta+\Ric)\geq\lambda$. Then for any $x_1,x_2\in M$ and $\varepsilon>0$, there is a complete manifold $(M',g')$ with the following properties:
    \begin{enumerate}[label={(\roman*)}, nosep]
        \item $M'$ is topologically obtained from $M$ by removing two small balls around $x_1,x_2$ and gluing back a tunnel $\mathbb S^{n-1}\times[0,1]$,
        \item $g'=g$ on $M\setminus\big(B(x_1,\varepsilon)\cup B(x_2,\varepsilon)\big)$,
        \item $\lambda_1(-\gamma\Delta_{g'}+\Ric_{g'})\geq\lambda-\varepsilon$.
    \end{enumerate}
\end{theorem}

The following are its immediate consequences:

\begin{corollary}\label{cor:ConnectedSum}
    Let $n\geq3$, $\gamma>\frac{n-1}{n-2}$, and $\lambda\in\RR$. For $i=1,2$, suppose $(M_i,g_i)$ are (possibly non-connected) complete $n$-manifolds satisfying $\lambda_1(-\gamma\Delta_{M_i}+\Ric_{M_i})>\lambda$. Then $M_1\#M_2$ admits a complete metric satisfying the same spectral Ricci lower bound.
\end{corollary}

\begin{corollary}\label{corSn-1S1}
    Let $n\geq3$, $\gamma>\frac{n-1}{n-2}$, $\lambda\in\RR$, and $M$ be a (possibly non-connected) complete $n$-manifold satisfying $\lambda_1(-\gamma\Delta+\Ric)>\lambda$. Then $M\#(\mathbb S^{n-1}\times \mathbb S^1)$ admits a complete Riemannian metric satisfying the same spectral Ricci lower bound.
\end{corollary}

In particular, this implies that the constant $\frac{n-1}{n-2}$ is sharp in Bour--Carron's Betti number bound \cite[Proposition 3.5]{BourCarron}. Indeed, as a consequence of the corollary above, for every $n\geq 3$, $\gamma>\frac{n-1}{n-2}$, and $k\geq 1$, $\#^k(\mathbb S^{n-1}\times \mathbb S^1)$ admits a metric with $\lambda_1(-\gamma\Delta+\mathrm{Ric})>0$.

The proof of Theorem \ref{thm:tunnel} follows from a tunnel construction in the manner of (but is technically different from) Gromov--Lawson and Schoen--Yau. The toy model comes from an observation of Bour--Carron: as a part of the rigidity result in \cite[Proposition 3.5]{BourCarron}, they observed that the metric $g=\mathrm{d}r^2+f(r)^2g_{\SS^{n-1}}$ and the function $u=f(r)^{2-n}$ together satisfy
\[
-\frac{n-1}{n-2}\Delta_g u+\Ric_g(\p_r,\p_r)\cdot u=0.
\]
This identity holds for an arbitrary $f$. In particular, one can choose $f$ so that $f(r)=|r|$ for all $|r|\geq\varepsilon$. With this choice, one produces a metric on $\RR^n\#\RR^n$ that satisfies $\lambda_1(-\frac{n-1}{n-2}\Delta+\Ric)\geq0$. This is the local picture of our construction in \cref{thm:tunnel}.

To reach the proof of \cref{thm:tunnel}, we have to globalize this model. Given the initial manifold $M$ and basepoints $x_1,x_2$, we remove two tiny balls and glue back a tiny tunnel. Our goal is to find a metric $g'$, and a function $\hat u>0$, so that $-\gamma\Delta_{g'}\hat u+\Ric_{g'}\cdot \hat u\geq(\lambda-\varepsilon)\hat u$. The easy step is to make $g'$ resemble the toy model $\mathrm{d}r^2+f(r)^2g_{\mathbb S^{n-1}}$ near the tunnel. It is less obvious how to merge the eigenfunction. The key observation is that in the toy model, we have $u\sim r^{2-n}$ near the origin. Therefore, as $\varepsilon\to0$, the model degenerates to two copies of $\RR^n$ with a Green's function on each of them. Inspired by this, we realize that one should try to merge a Green's function. Hence, on the initial manifold $M$ we will solve the equation
\[
-\gamma\Delta u+\Ric\cdot u=(\lambda-\varepsilon)u+(\delta_{x_1}+\delta_{x_2}),
\]
and the new function $\hat u$ is produced by merging this solution $u$ with the model solution $f^{2-n}$ near the tunnel.

It is interesting to notice that the warping factor $f$ here can be arbitrary as long as it is convex; see \eqref{eq:def_f} for the precise conditions. This is in contrast with the Gromov--Lawson construction for scalar curvature, where the warping factor is carefully designed to satisfy a certain differential inequality.

In analogy with the higher codimensional surgery result for scalar curvature (see \cite[Theorem A]{Gromov_FourLectures} or \cite[Corollary 6]{SchoenYauManuscripta}), it is interesting to ask if the condition $\lambda_1(-\gamma\Delta+\mathrm{Ric})>\lambda$ is stable under surgery of codimension $\geq k$, provided that $2<k\leq n-1$ and $\gamma>\frac{k-1}{k-2}$. Note that Theorem \ref{thm:tunnel} covers the $k=n$ case. We remark that surgery procedures preserving $\mathrm{Ric}>0$ have been studied, e.g., in \cite{ShaYang, Wraith}, and the recent \cite{Reiser}.

\begin{remark}\label{remLucio}
\cref{cor:ConnectedSum} holds true even if the condition
    \[
    \lambda_1(-\gamma\Delta_{M_i}+\mathrm{Ric}_{M_i})>\lambda
    \]
    is replaced with the weaker
    \[
    -\gamma\Delta_{M_i}+\mathrm{Ric}_{M_i}-\lambda \quad \text{is subcritical on $M_i$},
    \]
    see \cite[Definition 2.2]{CatinoMariMastroliaRoncoroni}.
    The proof follows from a minor modification of the strategy described above. For the sake of readability, let us focus on how the statement and proof of \cref{thm:tunnel} change. Then the improved version of \cref{cor:ConnectedSum} stated in this Remark follows.
    
    If $-\gamma\Delta+\mathrm{Ric}-\lambda$ is subcritical on $M$, then by definition there is a continuous function $W>0$ such that 
    \begin{equation}\label{eqn:One}
    -\gamma\Delta+\mathrm{Ric}-\lambda-W>0.
    \end{equation}
    Then, for every $\varepsilon>0$, by \cite[Theorem 2.3]{CatinoMariMastroliaRoncoroni}, we can find $0<u\in C^2(M\setminus\{x_1,x_2\})$ such that 
    \[
    -\gamma\Delta u + \mathrm{Ric}\cdot u = (W+\lambda-\varepsilon)u+\delta_{x_1}+\delta_{x_2}.
    \]
    The gluing procedure described above, and performed in \cref{prooftunnel}, allows us to produce a metric $\tilde g$ and a function $\tilde u$ such that, for $\varepsilon\ll 1$ we have
    \begin{equation}\label{eqn:Two}
    -\gamma\tilde\Delta\tilde u+\tilde{\mathrm{Ric}}\cdot\tilde u \geq (W+\lambda-2\varepsilon)\tilde u,
    \end{equation}
    close to the tunnel, see \cref{lem:EstimatesSpectralRicciModifiedMetric}. Furthermore, the metric and eigenfunction are unchanged away from the tunnel; hence, there we have
    \[
    -\gamma\tilde\Delta\tilde u+\tilde{\mathrm{Ric}}\cdot\tilde u \geq (W+\lambda)\tilde u.
    \]
    Therefore, since $W$ is continuous and positive, we get that $-\gamma\tilde\Delta+\tilde{\mathrm{Ric}}-\lambda$ is subcritical for small enough $\varepsilon$, as desired.
    %
\end{remark}


\textbf{Acknowledgments}. G.A. has been partially supported by the AMS-Simons Travel Grant, and the NSF DMS Grant No. 2505713. The authors thank Luciano Mari for suggesting the content of \cref{remLucio}.

\section{Proof of {\cref{thm:tunnel}}}\label{prooftunnel}

\subsection{Notations and setups}\label{sec:Prelimnaries}

    Let $x_1,x_2\in M$ and all the parameters $\gamma,\lambda,\varepsilon$ be as in the statement of \cref{thm:tunnel}. By applying \cite[Theorem 2.3 (i)\,$\Leftrightarrow$\,(iv)]{CatinoMariMastroliaRoncoroni}, and since $\mathrm{Ric}\in\mathrm{Lip}_{\mathrm{loc}}(M)$, we can find a Green's function $u\in C^{2,\alpha}(M\setminus\{x_1,x_2\})$ so that
    \begin{equation}\label{eqn:-gammadeltaucontrol}
        u>0,\qquad
        -\gamma\Delta u+\Ric u=(\lambda-\varepsilon/2)u+(\delta_{x_1}+\delta_{x_2}).
    \end{equation}
    Let $r_0\ll\min\big\{1,d(x_1,x_2),\varepsilon\big\}$ be a small radius to be chosen.
    
    In the polar coordinates near $U_i:=B(x_i,r_0)$, $i=1,2$, we write the metric and Green's function in the form
    \begin{equation}\label{eqn:MetricAndGreen}
    g|_{U_i}=\mathrm{d}r^2+r^2h_i(r,\th),\qquad u|_{U_i}=\frac{r^{2-n}\cdot w_i(r,\th)}{\gamma(n-2)|\mathbb S^{n-1}|},\qquad r\in(0,r_0),\ \ \th\in \mathbb S^{n-1},
    \end{equation}
    for some tensor $h_i(r,\th)$ and function $w_i(r,\th)$. For our convenience of the gluing procedure, let us reverse $r\mapsto-r$ in $U_1$, hence assume that $r\in(-r_0,0)$ for $g|_{U_1}$, and $r\in(0,r_0)$ for $g|_{U_2}$. Let $h_0$ denote the round metric on $\mathbb S^{n-1}$.
    
    We will need the following Lemma, whose proof is postponed to \cref{sec:AsymptoticGreen}. The asymptotics in \cref{lemma:asymp_h_w} actually hold in $C^k(h_0)$ for every $k\geq 0$, see the proofs of \cref{lemma-asymp:metric}, and \cref{lemma-asymp:green}. Anyway, the version stated here is sufficient for proving \cref{thm:tunnel}.
    \begin{lemma}\label{lemma:asymp_h_w}
        As $r\to0$, we have the following asymptotics:
        \begin{equation}\label{eq:asymp_h}
            \|h_i-h_0\|_{C^2(h_0)}=o(1),\quad \|\p_rh_i\|_{C^1(h_0)}=o(r^{-1}), \quad \|\p_r^2h_i\|_{C^0(h_0)}=o(r^{-2}),
        \end{equation}
        and
        \begin{equation}\label{eq:asymp_w}
            \|w_i-1\|_{C^2(h_0)}=o(1),\quad \|\p_rw_i\|_{C^0(h_0)}=o(r^{-1}),\quad \|\p_r^2w_i\|_{C^0(h_0)}=o(r^{-2}).
        \end{equation}
    \end{lemma}
    In the expressions \eqref{eq:asymp_h} and \eqref{eq:asymp_w}, the objects $h_i,w_i$ are viewed as $r$-parametrized families of tensors and functions on $\SS^{n-1}$, respectively.
    
    Let us now go on with the construction of the auxiliary functions we will need in the proof. From now on we say that $F=O(r_0^k)$ if $|F|\leq Cr_0^k$ for all $|r|\leq r_0$, where $C$ is a constant independent of $r_0$. We say that $F=o(r_0^k)$ if $|F|\leq r_0^k\varepsilon(r_0)$ for all $|r|\leq r_0$, where $\varepsilon(r_0)$ is a function with $\lim_{r_0\to0}\varepsilon(r_0)=0$.
    
    Fix a cutoff function $0\leq\eta\leq1$ so that $\eta|_{[-1/3,1/3]}\equiv0$ and $\eta|_{(-\infty,-2/3]\cup[2/3,+\infty)}\equiv1$. Define
    \begin{equation}\label{eqn:tildeh}
    \tilde h:=\left\{\begin{aligned}
        & h_1,\qquad -r_0<r<-2r_0/3, \\
        & \eta(r/r_0)h_1+\big(1-\eta(r/r_0)\big)h_0\qquad -2r_0/3\leq r\leq -r_0/3, \\
        & h_0\qquad -r_0/3<r<r_0/3, \\
        & \eta(r/r_0)h_2+\big(1-\eta(r/r_0)\big)h_0\qquad r_0/3\leq r\leq 2r_0/3, \\
        & h_2,\qquad 2r_0/3<r<r_0,
    \end{aligned}\right.
    \end{equation}
    and
    \begin{equation}\label{eqn:tildew}
    \tilde w:=\left\{\begin{aligned}
        & w_1,\qquad -r_0<r<-2r_0/3, \\
        & \eta(r/r_0)w_1+\big(1-\eta(r/r_0)\big)\qquad -2r_0/3\leq r\leq -r_0/3, \\
        & 1\qquad -r_0/3<r<r_0/3, \\
        & \eta(r/r_0)w_2+\big(1-\eta(r/r_0)\big)\qquad r_0/3\leq r\leq 2r_0/3, \\
        & w_2,\qquad 2r_0/3<r<r_0.
    \end{aligned}\right.
    \end{equation}
    These objects are a concatenation of $h_1,h_2$, and $w_1,w_2$. By \cref{lemma:asymp_h_w} and elementary computations, $\tth$ and $\tw$ satisfy the decay
    \begin{equation}\label{eq:asymp_tilde_h_w}
        \begin{aligned}
            & \|\tth-h_0\|_{C^2(h_0)}=o(1),\quad
            \|\p_r\tth\|_{C^1(h_0)}=o(r_0^{-1}),\quad
            \|\p^2_r\tth\|_{C^0(h_0)}=o(r_0^{-2}), \\
            & \|\tw-1\|_{C^2(h_0)}=o(1),\quad
            \|\p_r\tw\|_{C^0(h_0)}=o(r_0^{-1}),\quad
            \|\p^2_r\tw\|_{C^0(h_0)}=o(r_0^{-2}),
        \end{aligned}
    \end{equation}
    when $|r|\leq r_0$, for all sufficiently small $r_0$.
    
    Choose an even function $f\in C^\infty(\RR)$ so that
    \begin{equation}\label{eq:def_f}
        \left\{\begin{aligned}
            &\quad f>0,\quad f|_{(-\infty,-1]}=-x,\quad f|_{[1,\infty)}=x,\\
            &\  |f'|<1,\ f''>0\text{ in }(-1,1), \quad f'''<0 \text{ in }(0,1). \ 
        \end{aligned}\right\}
    \end{equation}
    
    Finally, define the new metric and eigenfunction
    \begin{equation}\label{eqn:gtiildeutilde}
    \tilde g:=\mathrm{d}r^2+r_0^2f(r/r_0)^2\tilde h,\qquad\tilde u:=\frac{r_0^{2-n}f(r/r_0)^{2-n}\tilde w}{\gamma(n-2)|\mathbb S^{n-1}|},\qquad -r_0\leq r\leq r_0.
    \end{equation}
    
    The crucial part of the proof is the following (which corresponds to \cref{thm:tunnel}(iii)):
    \begin{prop}\label{lem:EstimatesSpectralRicciModifiedMetric}
        Let $\gamma,\lambda,\varepsilon,\tilde g,\tilde u$ be as above. Then, for $r_0$ sufficiently small we have
        \[
        -\gamma\Delta_{\tilde g}\tilde u+\tilde\Ric\cdot\tilde u\geq(\lambda-\varepsilon)\tilde u\qquad\text{in}\ \ \big\{|r|\leq r_0\big\}.
        \]
    \end{prop}
    We now give the proof of \cref{thm:tunnel} using \cref{lem:EstimatesSpectralRicciModifiedMetric}. Then, in the next section, we go on with the proof of \cref{lem:EstimatesSpectralRicciModifiedMetric}.

    \begin{proof}[Proof of \cref{thm:tunnel}] {\ }
    Let $u$ be as in \eqref{eqn:-gammadeltaucontrol}, and let us write the metric $g$, and the function $u$ in coordinates as in \eqref{eqn:MetricAndGreen}, with $r_0>0$ small enough. Let $h_0$ be the standard metric on $\mathbb S^{n-1}$, and let $\tilde h,\tilde w$ be defined as in \eqref{eqn:tildeh}, and \eqref{eqn:tildew}, respectively. Finally, let $f$ be as in \eqref{eq:def_f}, and let the new metric $\tilde g$ and function $\tilde u$ be defined as in \eqref{eqn:gtiildeutilde}.

        Let $M'$ be obtained from $M$ by removing $B(x_1,r_0)\cup B(x_2,r_0)$ and gluing back the tunnel $D:=\SS^{n-1}\times[-r_0,r_0]$. Set $g'=\tilde g$ on $D$ and $g'=g$ on $M'\setminus D$. Then set $u'=\tilde u$ on $D$, and $u'=u$ on $M\setminus D$. By construction, $g',u'$ are smooth, and $g'$ satisfies Items (i) and (ii) of \cref{thm:tunnel}. Moreover, by \cref{lem:EstimatesSpectralRicciModifiedMetric} and \eqref{eqn:-gammadeltaucontrol} we have
        \[
        -\gamma\Delta_{g'}u'+\mathrm{Ric}_{g'}u'\geq (\lambda-\varepsilon)u',
        \]
        from which Item (iii) of \cref{thm:tunnel} follows, as desired.
    \end{proof}

    \subsection{Proof of \cref{lem:EstimatesSpectralRicciModifiedMetric}}

    In this section we will use the notation set up in \cref{sec:Prelimnaries}. For convenience, we denote
    \[\begin{aligned}
        & \text{Region I}:=\big\{\!-r_0\leq r\leq -2r_0/3\big\}\cup\big\{2r_0/3\leq r\leq r_0\big\}, \\
        & \text{Region II}:=\big\{|r|\leq 2r_0/3\big\}.
    \end{aligned}\]
    Note that $\eta(r/r_0)\equiv1$ in Region I.
    
    In the calculation below, we let $\tA,\tH$ be the second fundamental form and mean curvature of $\{r=\text{const}\}$ with respect to $\tg$. For the simplicity of formulas, we will always abbreviate $f:=f(r/r_0)$, $f':=f'(r/r_0)$, $f'':=f''(r/r_0)$. Note that $\p_rf=r_0^{-1}f'$. We calculate
    \begin{equation}\label{eq:tA}
        \tA=\frac12\p_r(r_0^2f^2\tth)=r_0ff'\tth+\frac12r_0^2f^2\p_r\tth,
    \end{equation}
    and
    \begin{equation}\label{eq:tH}
        \tH=\tr_{r_0^2f^2\tth}(\tA)=(n-1)\frac{f'}{r_0f}+\frac12\tr_{\tth}(\p_r\tth).
    \end{equation}
    In the following three lemmas, we compute the Ricci curvature of $\tilde g$. 

    \begin{lemma}\label{lem:Ricrr}
        We have
        \begin{align}
            & \text{In Region I:}\quad \tRic(\p_r,\p_r)=\Ric(\p_r,\p_r)-(n-1)\frac{f''}{r_0^2f}+o(r_0^{-2})\Big(\frac{rf'}{r_0f}-1\Big), \label{eq:Ric_rr_region1} \\
            & \text{In Region II:}\quad \tRic(\p_r,\p_r)=-(n-1)\frac{f''}{r_0^2f}+o(r_0^{-2}). \label{eq:Ric_rr_region2}
        \end{align}
    \end{lemma}
    \begin{proof}
        Recall the formula $\tRic(\p_r,\p_r)=-\p_r\tH-|\tA|^2$. Using \eqref{eq:tA} and \eqref{eq:tH}, this gives
        \begin{equation}\label{eq:Ric_rr_aux1}
            \tRic(\p_r,\p_r)=-(n-1)\frac{f''}{r_0^2f}-\frac{f'}{r_0f}\tr_\tth(\p_r\tth)-\frac14|\p_r\tth|^2_\tth-\frac12\p_r\big(\tr_\tth(\p_r\tth)\big).
        \end{equation}
        Combined with \eqref{eq:asymp_tilde_h_w} we get \eqref{eq:Ric_rr_region2}, since in Region II we have
        \[
        \frac{f'}{r_0f}=O(r_0^{-1}),\quad \p_r\tth=o(r_0^{-1}),\quad \p_r^2\tth=o(r_0^{-2}).
        \]
        In Region I we have $\tth=h_i$, so the original $\Ric(\p_r,\p_r)$ is obtained from \eqref{eq:Ric_rr_aux1} by plugging in $f(x)=x$ (notice that one should really plug in $f=r/r_0$, $f'=1$, $f''=0$ in \eqref{eq:Ric_rr_aux1}). This gives
        \begin{equation}\label{eq:Ric_rr_aux2}
            \Ric(\p_r,\p_r)=-\frac1r\tr_\tth(\p_r\tth)-\frac14|\p_r\tth|^2_\tth-\frac12\p_r\big(\tr_\tth(\p_r\tth)\big).
        \end{equation}
        Subtracting \eqref{eq:Ric_rr_aux1} with \eqref{eq:Ric_rr_aux2} and noticing that $|r^{-1}|\leq2r_0^{-1}$ in Region I, we obtain \eqref{eq:Ric_rr_region1}.
    \end{proof}

    In Lemma \ref{lem:Ricee} and \ref{lem:Ricre} below, we let $e$ be any $\tth$-unit tangent vector of $\mathbb S^{n-1}$, and $\te=e/r_0f$ (which is a $r_0^2f^2\tth$-unit vector).

    \begin{lemma}\label{lem:Ricee}
        We have:
        \begin{align}
            & \begin{aligned}
                &\text{In Region I:}\quad \tRic(\te,\te)=\Ric(e/r,e/r)-\frac{f''}{r_0^2f}+\frac{n-2}{r_0^2f^2}\big(1-(f')^2\big) \\
                &\hspace{180pt} +o(r_0^{-2})\Big(\frac{rf'}{r_0f}-1\Big)+o(r_0^{-2})\Big(\frac{r^2}{r_0^2f^2}-1\Big),
            \end{aligned}\label{eq:Ric_ee_region1} \\
            & \text{In Region II:}\quad \tRic(\te,\te)=-\frac{f''}{r_0^2f}+\frac{n-2}{r_0^2f^2}\big(1-(f')^2\big)+o(r_0^{-2}). \label{eq:Ric_ee_region2}
        \end{align}
    \end{lemma}
    \begin{proof}
        By Gauss equation
        \[\tRic(\te,\te)=\tR(\te,\p_r,\p_r,\te)+\Ric_{r_0^2f^2\tth}(\te,\te)-\tH\tA(\te,\te)+\tA^2(\te,\te).\]
        The first term is calculated by the variation of second fundamental form:
        \[   \tR(\te,\p_r,\p_r,\te)=\tA^2(\te,\te)-(\L_{\p_r}\tA)(\te,\te).
        \]
        Here we can further calculate by \eqref{eq:tA}
        \[(\L_{\p_r}\tA)(\te,\te)=\frac{f''}{r_0^2f}+\frac{(f')^2}{r_0^2f^2}+2\frac{f'}{r_0f}\p_r\tth(e,e)+\frac12\p_r^2\tth(e,e).\]
        Next, clearly $\Ric_{r_0^2f^2\tth}(\te,\te)=\Ric_\tth(\te,\te)=r_0^{-2}f^{-2}\Ric_\tth(e,e)$. Let $\{e_i\}$ be a $\tth$-orthonormal frame with $e=e_1$. We compute $\tA^2(\te,\te)$:
        \[\begin{aligned}
            \tA^2(\te,\te)
            &= \frac1{r_0^2f^2}\sum_i\tA(e_1,e_i)^2
            = \Big[\frac{f'}{r_0f}+\frac12\p_r\tth(e_1,e_1)\Big]^2+\sum_{i\geq 2}\Big[\frac12\p_r\tth(e_1,e_i)\Big]^2 \\
        &= \frac{(f')^2}{r_0^2f^2}+\sum_i\Big[\frac12\p_r\tth(e_1,e_i)\Big]^2+\frac{f'}{r_0f}\p_r\tth(e_1,e_1).
        \end{aligned}\]
        Combining all these results and using \eqref{eq:tA} and \eqref{eq:tH} to compute $\tH$, $\tA(\te,\te)$, we eventually get 
        \[\begin{aligned}
            \tRic(\te,\te)
            &= -\frac{f''}{r_0^2f}-(n-2)\frac{(f')^2}{r_0^2f^2}-\frac{f'}{r_0f}\Big[\frac{n-1}2\p_r\tth(e,e)+\frac12\tr_\tth(\p_r\tth)\Big] \\
            &\qquad +\frac1{r_0^2f^2}\Ric_\tth(e,e)-\frac12\p_r^2\tth(e,e)+2\sum_i\Big[\frac12\p_r\tth(e_1,e_i)\Big]^2.\\
            &\qquad -\frac{1}{4}\partial_r\tilde h(e,e)\mathrm{tr}_{\tilde h}(\partial_r\tilde h).
        \end{aligned}\]
        In Region II, the result follows by bounding the error terms using \eqref{eq:asymp_tilde_h_w}. In Region I, again, setting $f(x)=x$ this formula yields $\Ric(e/r,e/r)$. The result follows by extracting the difference and further noticing $\Ric_\tth(e,e)=(n-2)+o(1)$ by \eqref{eq:asymp_tilde_h_w}.
    \end{proof}

    \begin{lemma}\label{lem:Ricre}
        We have
        \begin{align}
            & \text{In Region I:}\quad \tRic(\p_r,\te)=\Ric(\p_r,e/r)+o(r_0^{-2})\Big(\frac{r}{r_0f}-1\Big), \label{eq:Ric_re_region1} \\
            & \text{In Region II:}\quad \tRic(\p_r,\te)=o(r_0^{-2}). \label{eq:Ric_re_region2}
        \end{align}
    \end{lemma}
    \begin{proof}
        Let $\{e_i\}$ be a $\tth$-orthonormal frame with $e=e_1$ as in \cref{lem:Ricee}. We have
        \[\begin{aligned}
            \tRic(\p_r,\te)
            &=\sum_i \frac1{r_0^3f^3}\tR(\p_r,e_i,e_i,e)
            = \sum_i\frac1{r_0^3f^3}\Big[\D_e^{r_0^2f^2\tth}\tA(e_i,e_i)-\D_{e_i}^{r_0^2f^2\tth}\tA(e,e_i)\Big].
        \end{aligned}\]
        By \eqref{eq:tA} we have $\D^{r_0^2f^2\tth}\tA=\D^\tth\tA=\D^\tth\big(\frac12r_0^2f^2\p_r\tth\big)$. So
        \[\tRic(\p_r,\te)=\sum_i\frac1{2r_0f}\Big[\D_e^\tth(\p_r\tth)(e_i,e_i)-\D_{e_i}^\tth(\p_r\tth)(e,e_i)\Big].\]
        Again, setting $f(x)=x$ this formula yields $\Ric(\partial_r,e/r)$. In Region II the right hand side is bounded by \eqref{eq:asymp_tilde_h_w}, and in Region I the result follows as well.
    \end{proof}

\begin{corollary}
    We have
    \begin{align}
            & \text{In Region I:}\quad \tRic\geq\Ric-(n-1)\frac{f''}{r_0^2f}-o(r_0^{-2})\Big|\frac{rf'}{r_0f}-1\Big|-o(r_0^{-2})\Big|\frac{r}{r_0f}-1\Big|, \label{eq:Ric_region1} \\
            & \text{In Region II:}\quad \tRic\geq-(n-1)\frac{f''}{r_0^2f}-o(r_0^{-2}). \label{eq:Ric_region2}
        \end{align}
\end{corollary}
\begin{proof}
    It is an immediate corollary of \cref{lem:Ricrr}, \cref{lem:Ricee}, and \cref{lem:Ricre}, the fact that $|f'|<1$, $f''>0$ by \eqref{eq:def_f}, and the inequality $\big|\frac{r^2}{r_0^2f^2}-1\big|\leq O(1)\big|\frac{r}{r_0f}-1\big|$.
\end{proof}

We next compute the quantity $\frac{\tDelta\tu}{\tu}$.

    \begin{lemma}\label{lem:AsymptoticsDeltauu}
        We have
        \begin{align}
            & \text{In Region I:}\quad \frac{\tDelta\tu}{\tu}=\frac{\Delta u}u+(2-n)\frac{f''}{r_0^2f}+o(r_0^{-2})\Big(\frac{rf'}{r_0f}-1\Big)+o(r_0^{-2})\Big(\frac{r}{r_0f}-1\Big),\label{eq:Delta_u_region1} \\
            & \text{In Region II:}\quad \frac{\tDelta\tu}{\tu}=(2-n)\frac{f''}{r_0^2f}+o(r_0^{-2}). \label{eq:Delta_u_region2}
        \end{align}
    \end{lemma}
    \begin{proof}
        In this proof, for ease of notation, we will denote $\tu_r:=\p_r \tu$, and $\tu_{rr}:=\p_r^2\tu$. Note that $\tDelta\tu=\tu_{rr}+\tH\tu_r+\frac1{r_0^2f^2}\Delta_\tth\tu$. Thus one can directly compute
        \begin{align}
            \tu^{-1}\tu_r &= (2-n)\frac{f'}{r_0f}+\frac{\tw_r}{\tw}, \\
            \tu^{-1}\tu_{rr} &= (2-n)\frac{f''}{r_0^2f}+(2-n)(1-n)\frac{(f')^2}{r_0^2f^2}+2(2-n)\frac{f'}{r_0f}\frac{\tw_r}{\tw}+\frac{\tw_{rr}}{\tw}.
        \end{align}
        So
        \begin{equation}\label{eq:Delta_u_aux1}
            \begin{aligned}
                \frac{\tDelta\tu}{\tu} &= (2-n)\frac{f''}{r_0^2f}+\frac{\tw_{rr}}{\tw}+\frac12\tr_\tth(\p_r\tth)\frac{\tw_r}{\tw}+\frac1{r_0^2f^2}\frac{\Delta_\tth\tw}{\tw} \\
                &\qquad +\frac{f'}{r_0f}\Big[(3-n)\frac{\tw_r}{\tw}+\frac{2-n}2\tr_\tth(\p_r\tth)\Big].
            \end{aligned}
        \end{equation}
        In Region II, we recall that $f,f^{-1},f'$ are all bounded and, using \eqref{eq:asymp_tilde_h_w}
        \begin{equation}\label{eq:Delta_u_aux2}
            \frac{\tw_r}{\tw}=o(r_0^{-1}),\quad \frac{\tw_{rr}}{\tw}=o(r_0^{-2}),\quad \frac{\Delta_{\tth}\tw}{\tw}=o(1),\quad \tr_{\tth}\big(\p_r\tth\big)=o(r_0^{-1}).
        \end{equation}
        Inserting these into \eqref{eq:Delta_u_aux1}, we obtain \eqref{eq:Delta_u_region2}.
        
        In Region I, recall that $\tilde w=w_i$ and $\tilde h=h_i$, so $\Delta u/u$ is recovered from this formula by inserting $f(x)=x$. Using this fact, inserting \eqref{eq:Delta_u_aux2}, and recalling that $|r^{-1}|\leq 2r_0^{-1}$ in Region I, we obtain \eqref{eq:Delta_u_region1}. 
    \end{proof}
    
    \begin{proof}[Proof of {\cref{lem:EstimatesSpectralRicciModifiedMetric}}] {\ }
        First consider Region I.
        Combining \eqref{eq:Ric_region1} with \eqref{eq:Delta_u_region1}, the sought inequality $-\gamma\tDelta\tu+\tRic\cdot\tu\geq(\lambda-\varepsilon)\tu$ is implied by 
        \[
        \begin{aligned}
            & -\gamma\frac{\Delta u}{u}+\Ric+\Big[\gamma(n-2)-(n-1)\Big]\frac{f''}{r_0^2f}-o(r_0^{-2})\Big|\frac{rf'}{r_0f}-1\Big|-o(r_0^{-2})\Big|\frac{r}{r_0f}-1\Big|\geq\lambda-\varepsilon.
        \end{aligned}
        \]
        Since on Region I we know that $-\gamma\Delta u+\Ric u=(\lambda-\varepsilon/2)u$, it suffices to show that
        \begin{equation}\label{eq:regionI_aux1}
            \Big[\gamma(n-2)-(n-1)\Big]\frac{f''}{r_0^2f}\geq o(r_0^{-2})\Big|\frac{rf'}{r_0f}-1\Big|+o(r_0^{-2})\Big|\frac{r}{r_0f}-1\Big|.
        \end{equation}
        Since $\gamma>\frac{n-1}{n-2}$, the left hand side has a positive coefficient. By the algebraic \cref{lemma:property_of_f} below, and the fact that $f$ is even, it follows that \eqref{eq:regionI_aux1} holds in Region I as long as $r_0$ is sufficiently small.
    
        Next consider Region II. Using \eqref{eq:Ric_region2} and \eqref{eq:Delta_u_region2}, we find that $-\gamma\tDelta\tu+\tRic\tu\geq(\lambda-\varepsilon)\tu$ is implied by
        \begin{equation}\label{eq:regionII_aux1}
            \big[\gamma(n-2)-(n-1)\big]\frac{f''}{r_0^2f}-o(r_0^{-2})\geq\lambda-\varepsilon.
        \end{equation}
        Since $r/r_0\in(-2/3,2/3)$ in Region II, \eqref{eq:regionII_aux1} holds for sufficiently small $r_0$ due to the strict convexity of $f$. Thus \cref{lem:EstimatesSpectralRicciModifiedMetric} is proved in Region II as well, as desired.
    \end{proof}
    \begin{lemma}\label{lemma:property_of_f}
        Let $f\in C^\infty(\mathbb R)$ be an even function that satisfies \eqref{eq:def_f}. Then 
        \begin{equation}
            \frac{f''}{f}\geq \frac{1}{2}\Big|\frac{xf'}{f}-1\Big|+\frac{1}{2}\Big|\frac{x}f-1\Big|\qquad\text{for all $x\in[1/2,1]$.}
        \end{equation}
    \end{lemma}
    \begin{proof}[Proof of \cref{lemma:property_of_f}]
        Since $f(x)>\max\{x,0\}$ and $xf'(x)<f(x)$ in $(-1,1)$, it suffices to show
        \[
        f''(x)\geq f(x)-xf'(x)\quad\text{and}\quad f''(x)\geq f(x)-x,\qquad\forall\,x\in(1/2,1).
        \]
        For the first inequality, note that for every $x\in (1/2,1)$,
        \[
        f(x)-xf'(x)=\int_x^1tf''(t)\,\mathrm{d}t\leq \int_x^1 f''(t)\,\mathrm{d}t\leq f''(x),
        \]
        where we used that $f''>0$ and $f'''<0$ on $(1/2,1)$. For the second inequality notice that $\forall\,x\in (1/2,1)$ we have 
        \[
        f(x)-x \leq f(x)-xf'(x) \leq f''(x),
        \]
        where we used $f'\leq 1$ and the first inequality.
    \end{proof}

\section{Asymptotics of Green's function and the metric}\label{sec:AsymptoticGreen}

The aim of this section is to prove Lemma \ref{lemma-asymp:metric} and \cref{lemma-asymp:green}, which together imply \cref{lemma:asymp_h_w}. We fix the following setup: $x_0\in M$ is a base point, and $R>0$ is smaller than the injectivity radius of $x_0$. Set the radius function $r:=d(\cdot,x_0)$. Let $h_0$ be the round metric on $\SS^{n-1}$.

\begin{lemma}\label{lemma-asymp:metric}
    Suppose that in $B(x_0,R)$, the metric is expressed in polar coordinates as $g=\mathrm{d}r^2+r^2h(r,\th)$. Then we have the asymptotics at $x_0$
    \begin{equation}
        \|h-h_0\|_{C^2(h_0)}=o(1),\quad \|\p_rh\|_{C^1(h_0)}=o(r^{-1}),\quad \|\p_r^2h\|_{C^0(h_0)}=o(r^{-2}).
    \end{equation}
\end{lemma}
\begin{proof}
    In the corresponding Cartesian coordinates, we have
    \begin{equation}\label{eq-asymp:metric0}
        \lim_{\lambda\to0}\lambda^{-2}\big[(x\mapsto\lambda x)^*g\big]=g_0\qquad\text{in}\ \ C^\infty_{\loc}(\RR^n),
    \end{equation}
    where $g_0$ is the Euclidean metric. Note that the map $[1/2,2]\times\SS^{n-1}\to\RR^n:(r,\th)\mapsto r\th$ is a diffeomorphism onto its image. So we can pull back \eqref{eq-asymp:metric0} along this map and obtain
    \begin{equation}\label{eq-asymp:metric1}
        \lim_{\lambda\to0}\big[\mathrm{d}r^2+r^2h(\lambda r,\th)\big]=\mathrm{d}r^2+r^2h_0(\th)\qquad\text{in}\ \ C^\infty\big([1/2,2]\times\SS^{n-1}\big).
    \end{equation}
    Taking successive $r$-derivatives, we obtain
    \begin{align}
        & \lim_{\lambda\to0}\big[2rh(\lambda r,\th)+\lambda r^2(\p_rh)(\lambda r,\th)\big]=2rh_0(\th) \label{eq-asymp:metric2}\\
        & \lim_{\lambda\to0}\big[2h(\lambda r,\th)+4\lambda r(\p_rh)(\lambda r,\th)+\lambda^2r^2(\p_r^2h)(\lambda r,\th)\big]=2h_0(\th) \label{eq-asymp:metric3}
    \end{align}
    in $C^\infty\big([1/2,2]\times\SS^{n-1}\big)$.
    
    Insert $r=1$ in each of the equations. From \eqref{eq-asymp:metric1} we obtain $\|h(\lambda,\cdot)-h_0(\cdot)\|_{C^k(h_0)}=o(1)$ when $\lambda\to0$, for all $k\geq0$. Then inserting this into \eqref{eq-asymp:metric2} we obtain $\|(\p_rh)(\lambda,\cdot)\|_{C^k(h_0)}=o(\lambda^{-1})$, for all $k\geq 0$. Inserting these into \eqref{eq-asymp:metric3} we finally obtain $\|(\p_r^2h)(\lambda,\cdot)\|_{C^k(h_0)}=o(\lambda^{-2})$ for all $k\geq 0$, as desired.
\end{proof}

The proof of the following lemma is similar to that of \cite[Theorem 6.1]{MRS22} or \cite[Appendix A]{BMRSX25}; see also \cite[Appendix A]{DruetHebeyRobert}.

\begin{lemma}\label{lemma-asymp:green}
    Suppose $n\geq3$, $a>0$. Let $f\in \mathrm{Lip}_{\mathrm{loc}}(B(x_0,R))$, and $u>0\in C^{2,\alpha}(B(x_0,R))$ solve $-\Delta u+fu=a\delta_{x_0}$. Define $b:=\frac{a}{(n-2)|\mathbb S^{n-1}|}$, and set $w(r,\theta):=\frac {u(r,\theta)}{br^{2-n}}$ in polar coordinates $(r,\theta)$, where $\theta\in\mathbb S^{n-1}$. Then when $r\to0$ we have
    \begin{equation}\label{eq-asymp:result2}
        \|w-1\|_{C^2(h_0)}=o(1),\quad \|\p_rw\|_{C^0(h_0)}=o(r^{-1}),\quad \|\p_r^2w\|_{C^0(h_0)}=o(r^{-2}).
    \end{equation}
\end{lemma}
\begin{proof}
    For simplicity, we denote $B(r):=B(x_0,r)$. Since the statement only involves the asymptotic behavior near $x_0$, we may decrease $R$ and assume that $f,u$ are defined in a neighborhood of $B(R)$.
    
    We may further decrease $R$ and assume that: $R$ is less than the injectivity radius at $x_0$, and $|\text{sec}|<K$ in $B(R)$ for some $K>0$. Set $F:=1+\sup_{B(R)}|f|$. Denote by $M_K$ the space form of constant curvature $K$, and $\bG=\bG(r)>0$ be a radial solution of $\Delta^{M_K}\bG=-F\bG-a\delta_o$ in the ball $B(2R)\subset M_K$. For $R$ sufficiently small, such function exists and is decreasing with respect to $r$. By elementary computations, we have
    \begin{equation}\label{eq-asymp:model_G}
        \bG(r)=\frac{a}{(n-2)|\mathbb S^{n-1}|}r^{2-n}+o(r^{2-n})\qquad\text{when}\ \ r\to0.
    \end{equation}
    By Laplacian comparison, on the original manifold $M$ we have
    \[\Delta\bar G(r)\leq f\bar G(r).\]
    Also, by \cite[Theorem 1]{Serrin}, which can be applied because $u$ is not locally bounded around $x_1$ and $x_2$ thanks to the removability theorem in \cite[Theorem 10]{SerrinActa1}, there is a constant $C$ such that
    \begin{equation}\label{eq-asymp:rough}
        C^{-1}r^{2-n}\leq u\leq Cr^{2-n}\qquad\text{in}\ \ B(R)\setminus\{x_0\}.
    \end{equation}

    Note that $u$ is a positive Green's function for the elliptic operator $-\Delta+f$. Hence, there is a solution $v$ of the boundary value problem
    \[\left\{\begin{aligned}
        & \Delta v=fv\qquad\text{in }B(R), \\
        & v=\sup_{\p B(R)}u\qquad\text{on }\p B(R).
    \end{aligned}\right.\]
    For $s\in(0,R]$, define
    \[\bar\gamma(s)=\max\Big\{\frac{u(x)-v(x)}{\bar G(r(x))}: s\leq r(x)\leq R\Big\}.\]
    Then $\bar\gamma(R)=0$ and $\bar\gamma(s)$ is non-increasing for all $s<R$, and positive for small enough $s$.
    
    \vspace{3pt}
    \noindent\textbf{Claim.} For $s\in(0,R]$ small enough, the maximum in $\bar\gamma(s)$ is attained on $\p B(s)$.

    \vspace{3pt}
    \noindent\textit{Proof.} Fix $s$. Set $q(x):=u(x)-v(x)-\bar\gamma(s)\bar G(r(x))$. By construction, $q\leq0$ on $B(R)\setminus B(s)$. By the definition of $v$, and since $\bar\gamma(s)>0$ for $s\in (0,R]$ small enough, we have $q<0$ on $\p B(R)$. Also, note that $\Delta q\geq fu-fv-\bar\gamma(s)f\bG(r)=fq$ on $B(R)\setminus B(s)$. By the strong maximum principle, $q<0$ in the interior of $B(R)\setminus B(s)$.

    \vspace{3pt}
    Back to the main proof. Define
    \begin{equation}\label{eq-asymp:gamma_star}
        \gamma^*=\limsup_{x\to x_0}\frac{u(x)}{\bar G(r(x))}=\lim_{s\to0}\bar\gamma(s),
    \end{equation}
    which is finite due to \eqref{eq-asymp:model_G} \eqref{eq-asymp:rough}.

    For each $\lambda<1$, set $g_\lambda:=\lambda^{-2}g$ and $u_\lambda:=\lambda^{n-2}u$. Then we have a smooth convergence induced by the exponential map (where $g_0$ is the Euclidean metric):
    \begin{equation}\label{eq-asymp:blow_up_metric}
        (B(R),g_\lambda,x_0)\xrightarrow{C^\infty_{\loc}}(\RR^n,g_0,0).
    \end{equation}
    By \eqref{eq-asymp:rough} and interior elliptic estimates, for each sequence $\lambda_i\to0$ there is a subsequence such that: along with the metric convergence \eqref{eq-asymp:blow_up_metric} it holds
    \begin{equation}\label{eq-asymp:blow_up_u}
        u_{\lambda_i}\xrightarrow{C^{2,\alpha}_{\loc}}u_\infty\in C^{2,\alpha}_{\loc}(\RR^n\setminus\{0\}).
    \end{equation}
    Note that each $u_\lambda$ solves $\Delta_{g_\lambda}u_\lambda=\lambda^2\Delta_gu_\lambda=\lambda^2fu_\lambda$. Therefore, $u_\infty$ is harmonic.

    By \eqref{eq-asymp:gamma_star} \eqref{eq-asymp:blow_up_metric} \eqref{eq-asymp:blow_up_u} and \eqref{eq-asymp:model_G}, for each $w\in\RR^n\setminus\{0\}$ we have
    \begin{equation}\label{eq-asymp:uinfty_1}
        \begin{aligned}
            u_\infty(w) &= \lim_{i\to\infty}\lambda_i^{n-2}u(\exp_{x_0}(\lambda_iw)) \\
            &\leq \lim_{i\to\infty}\lambda_i^{n-2}\Big(\bar\gamma(\lambda_i|w|)\bar G(\lambda_i|w|)+\|v\|_{L^\infty}\Big)
            = \hat\gamma|w|^{2-n},
        \end{aligned}
    \end{equation}
    where $\hat\gamma=\gamma^*a/((n-2)|\mathbb S^{n-1}|)$.  On the other hand, recall from the Claim  that each $\bar\gamma(\lambda_i)$ is realized at some point $\exp_{x_0}(\lambda_iw_i)$, $|w_i|=1$. There is a subsequence (which we do not relabel) so that $w_i\to w_\infty$. Computing as in \eqref{eq-asymp:uinfty_1}, we have $u_\infty(w_\infty)=\hat\gamma|w_\infty|^{2-n}$. By the strong maximum principle, we obtain
    \[u_\infty(w)=\hat\gamma|w|^{2-n}\qquad\text{on\ \ }\RR^n\setminus\{0\}.\]
    Since the sequence $\lambda_i$ is arbitrary, we in fact have $\lim_{\lambda\to0}u_\lambda=u_\infty$ in $C^{2,\alpha}_{\loc}(\RR^n\setminus\{0\})$.
    
    Let us define $\hat w(r,\theta):=\frac{u(r,\theta)}{\hat \gamma r^{2-n}}$. Thus, for the same reason as in \cref{lemma-asymp:metric}, we get
    \[
    \lim_{\lambda\to 0} \big[r^{2-n}\cdot \hat\gamma\cdot\hat w(\lambda r,\theta)\big] = \lim_{\lambda\to 0}\frac{u(\lambda r,\theta)}{\lambda^{2-n}} = \hat\gamma r^{2-n}\qquad\text{in}\ \ C^\infty\big([1/2,2]\times\SS^{n-1}\big).
    \]
    Thus, taking succesive $r$-derivatives, and evaluating at $r=1$ as in \cref{lemma-asymp:metric}, we get
    \begin{equation}\label{eq-asymp:result2hat}
        \|\hat w-1\|_{C^2(h_0)}=o(1),\quad \|\p_r\hat w\|_{C^0(h_0)}=o(r^{-1}),\quad \|\p_r^2\hat w\|_{C^0(h_0)}=o(r^{-2}).
    \end{equation}
    
    Let us show that $\hat\gamma=b$. By \eqref{eq-asymp:result2hat} we get that 
    \begin{equation}\label{eqn:FinalAsymp}
    u=\hat\gamma r^{2-n}(1+o(1)), \qquad \nabla u=\hat\gamma(2-n)r^{1-n}\nabla r(1+o(1)).
    \end{equation}
    Now let us use integration by parts. Take radii $0<\rho<R$ and a cutoff function $\varphi\in C^\infty_0(B(\rho))$ with $\varphi(x_0)=1$, $0\leq\varphi\leq1$. Using $-\Delta u+fu=a\delta_{x_0}$ and \eqref{eqn:FinalAsymp} we have
    \[\begin{aligned}
        a &= \int_{B(\rho)}fu\varphi+\metric{\D u}{\D\varphi}=o_\rho(1)+\int_{B(\rho)}\hat\gamma(2-n)r^{1-n}\metric{\D r}{\D\varphi} \\
        &= o_\rho(1)+\int_{\mathbb S^{n-1}}\int_0^\rho\hat\gamma(2-n)r^{1-n}\frac{\p\varphi}{\p r}\cdot r^{n-1}\big(1+o_\rho(1)\big)\,\mathrm{d}r\,\mathrm{d}V_{S^{n-1}} \\
        &= o_\rho(1)-\hat\gamma(2-n)|\mathbb S^{n-1}|\big(1+o_\rho(1)\big).
    \end{aligned}\]
    Hence $\hat\gamma=\frac{a}{(n-2)|\mathbb S^{n-1}|}=b$ and \eqref{eq-asymp:result2} directly follows from \eqref{eq-asymp:result2hat}.
\end{proof}

\printbibliography[heading=bibintoc,title={Bibliography}]

\end{document}